\documentclass[12pt]{amsart}
\usepackage{amsmath, amssymb, amsthm, amsfonts}
\usepackage{tikz}
\usepackage{geometry}
\usepackage{url}

\newtheorem{introthm}{Theorem}
\newtheorem{theorem}{Theorem}[section]

\newtheorem{lemma}[theorem]{Lemma}

\theoremstyle{definition}
\newtheorem{remark}[theorem]{Remark}

\newtheorem{definition}[theorem]{Definition}

\def\IR{\mathbb{R}}
\def\sm{\setminus}

\def\eps{\varepsilon}
\def\al{\alpha}
\def\be{\beta}
\def\de{\delta}
\def\la{\lambda}
\def\ro{\varrho}
\def\si{\sigma}

\DeclareMathOperator{\Aut}{Aut}
\DeclareMathOperator{\var}{var}
\DeclareMathOperator{\cov}{cov}
\DeclareMathOperator{\corr}{corr}

\newcommand{\mathdef}{\stackrel{\textrm{\scriptsize def}}{=}}
\def\eq{=}

\renewcommand{\vec}[1]{\mathbf{#1}}

\begin{document}

\title[Invariant Gaussian processes and independent sets]
{Invariant Gaussian processes and independent sets\\
on regular graphs of large girth}

\author[Cs\'oka]{Endre Cs\'oka}
\address{Mathematics Institute, 
University of Warwick} 
\email{csoka.endre@renyi.mta.hu}

\author[Gerencs\'er]{Bal\'azs Gerencs\'er}
\address{Alfr\'ed R\'enyi Institute of Mathematics, 
Hungarian Academy of Sciences} 
\email{gerencser.balazs@renyi.mta.hu}

\author[Harangi]{Viktor Harangi}
\address{Department of Mathematics, 
University of Toronto}
\email{harangi@math.toronto.edu}

\author[Vir\'ag]{B\'alint Vir\'ag}
\address{Department of Mathematics, 
University of Toronto}
\email{balint@math.toronto.edu}



\date{}

\begin{abstract}
We prove that every $3$-regular, $n$-vertex simple graph
with sufficiently large girth
contains an independent set of size at least $0.4361 n$.
(The best known bound is $0.4352 n$.)
In fact, computer simulation suggests that the bound our method provides
is about $0.438 n$. 

Our method uses invariant Gaussian processes on the $d$-regular tree
that satisfy the eigenvector equation at each vertex for a certain eigenvalue $\la$.
We show that such processes can be approximated by i.i.d.\ factors
provided that $|\la| \leq 2\sqrt{d-1}$. We then use these approximations
for $\la = -2\sqrt{d-1}$ to produce factor of i.i.d.\ independent sets on regular trees.
\end{abstract}
%
%
%
%
%

\maketitle


\section{Introduction}

An \emph{independent set} is a set of vertices in a graph, no two of which are adjacent.
The \emph{independence ratio} of a graph is the size of its largest independent set
divided by the total number of vertices.
Let $d \geq 3$ be an integer and suppose that
$G$ is a $d$-regular finite graph with sufficiently large girth,
that is, $G$ does not contain cycles shorter than a sufficiently large given length.
In other words, $G$ locally looks like a $d$-regular tree.
What can we say about the independence ratio of $G$?

In a regular (infinite) tree every other vertex can be chosen,
so one is tempted to say that the independence ratio should tend to $1/2$
when the girth goes to infinity. This is not the case, however:
Bollob\'as \cite{Bo_ind_set} showed that (uniform) random $d$-regular graphs
have essentially large girth (i.e., the number of short cycles is small)
and their independence ratios are bounded away from $1/2$ with high probability.
Asymptotically (as $d \to \infty$)
the independence ratio of the random $d$-regular graph is $2 (\log d) / d$
(the lower bound is due to Frieze and {\L}uczak \cite{frieze_random_regular}).
The best known upper bound for random 
$3$-regular graphs is $0.45537$ due to McKay \cite{mckay}, 
who sharpened \cite{Bo_ind_set}. 

%
Shearer \cite{shearer} showed that 
for any triangle-free graph with average degree $d$
the independence ratio is at least
$$ \frac{d \log d - d + 1}{(d-1)^2} .$$
For regular graphs of large girth, 
Shearer himself found an improvement \cite{shearer2}. 
Lauer and Wolmard further improved that bound for $d \geq 7$ 
by analyzing a simple greedy algorithm \cite{lauer_wormald}. 
All of these bounds are the same asymptotically: $(\log d)/d$. 
For small values of $d$ more sophisticated algorithms 
have been analyzed using computer-assisted proofs: 
in his thesis Hoppen presents an approach that outdoes
the above-mentioned bounds when $d \leq 10$
\cite[Table 5.3.1]{hoppen_thesis}.
For $d=3$ Kardo\v{s}, Kr\'al and Volec improved Hoppen's method
and obtained the bound $0.4352$ \cite{cubic}.
Compare this to McKay's upper bound $0.45537$.

The above lower bounds are based on local improvements of 
the standard greedy algorithm. Our main theorem is based on 
a different approach: we use Gaussian wave functions to find independent sets. 
\begin{introthm} \label{thm:3reg}
Every $3$-regular graph with sufficiently large girth
has independence ratio at least $0.4361$.
\end{introthm}
A related problem is finding induced bipartite subgraphs
with a lot of vertices. (Equivalently, we are looking for
two disjoint independent sets with large total size.)
This problem was studied for random $3$-regular graphs in \cite{hladky,hladky_bachelor}.
%
\begin{introthm} \label{thm:ind_bipartite}
Every $3$-regular graph with sufficiently large girth
has an induced subgraph that is bipartite and that contains
at least a 
$$ 1 - \frac{3}{4\pi} \arccos \left( \frac{5}{6} \right) > 0.86 $$
fraction of the vertices. 
\end{introthm}

To illustrate our strategy for proving Theorem \ref{thm:3reg}, 
suppose that there is a real number assigned to 
each vertex of $G$, called the value of the vertex. 
We always get an independent set by choosing those vertices
having larger values than each of their neighbors.
If we assign these values to the vertices in some random manner,
then we get a random independent set.
If the expected size of this random independent set can be computed,
then it gives a lower bound on the independence ratio.
In many cases, the probability that a given vertex is chosen is
the same for all vertices, in which case this probability itself is a lower bound.

The idea is to consider a random assignment that is
\emph{almost} an eigenvector (with high probability)
with some negative eigenvalue $\la$.
Then we expect many of the vertices with positive values to be chosen.
The spectrum of the $d$-regular tree is $[-2\sqrt{d-1}, 2\sqrt{d-1}]$,
so it is reasonable to expect that
we can find such a random assignment for $\la = -2\sqrt{d-1}$.
As we will see, the approach described above can indeed be carried out,
and it produces a lower bound 
$$ \frac{1}{2} - \frac{3}{4\pi} \arccos\left( 
\frac{1+2\sqrt{2}}{4} \right) \approx 0.4298 $$ 
in the $d=3$ case. This natural bound is already sharper 
than all previous bounds that are not computer-assisted. 

Using the same random assignment but a more sophisticated way
to choose the vertices for our independent set provides a better bound.
We fix some threshold $\tau \in \IR$, and 
we only keep those vertices that are below this threshold.
We choose $\tau$ in such a way that the components of the remaining vertices
are small with high probability. We omit the large components
and we choose an independent set from each of the small components.
(Note that the small components are all trees provided that
the girth of the original graph is large enough.
Since trees are bipartite, they have an independent set
containing at least half of the vertices.)
We simulated this random procedure on computer and
the probability that a given vertex is in the independent set
seems to be above $0.438$ in the $3$-regular case.
The best bound we could obtain with a rigorous proof is $0.4361$.
The proof is computer-assisted in the sense
that we used a computer to find certain numerical integrals.

If one wants to avoid using computers,
then one can set $\tau = 0$ and use simple estimates
to obtain a bound as good as $0.43$.
(The best previous bound obtained without the use of computers
is $0.4139$ and is due to Shearer, see \cite[Table 1]{lauer_wormald}.)
Note that one can also choose an independent set
from the vertices above the threshold in the same manner.
This other independent set is clearly disjoint from the first one
and has the same expected size when $\tau = 0$.
This is how Theorem \ref{thm:ind_bipartite} will be obtained.

\subsection*{Random processes on the regular tree}

Instead of working on finite graphs with large girth,
it will be more convenient for us to consider the regular (infinite) tree
and look for independent sets on this tree that are i.i.d.\ factors.

Let $T_d$ denote the $d$-regular tree for some positive integer $d \geq 3$,
$V(T_d)$ is the vertex set, and $\Aut(T_d)$ is the group of graph automorphisms of $T_d$.
%
Suppose that we have independent standard normal random variables $Z_v$
assigned to each vertex $v \in V(T_d)$.
We call an instance of an assignment a configuration.
A \emph{factor of i.i.d.\ independent set} is a random independent set
that is obtained as a measurable function of the configuration
and that commutes with the natural action of $\Aut(T_d)$.
By a \emph{factor of i.i.d.\ process} we mean random variables $X_v$, $v \in V(T_d)$
that are all obtained as measurable functions of the random variables $Z_v$
and that are $\Aut(T_d)$-invariant
(that is, they commute with the natural action of $\Aut(T_d)$).
Actually, in this paper we will only consider
\emph{linear factor of i.i.d.\ processes} defined as follows.
\begin{definition} \label{def:linear_factor}
We say that a process $X_v$, $v \in V(T_d)$ is a \emph{linear factor} of the
i.i.d.\ process $Z_v$ if there exist real numbers $\al_0, \al_1, \ldots$ such that
\begin{equation} \label{eq:lin}
X_v = \sum_{u \in V(T_d)} \al_{ d(v,u) } Z_u =
\sum_{k=0}^\infty \sum_{u: d(v,u) = k} \al_k Z_u ,
\end{equation}
where $d(v,u)$ denotes the distance between the vertices $v$ and $u$ in $T_d$.
Note that the infinite sum in \eqref{eq:lin} converges almost surely
if and only if $\al_0^2 + \sum_{k=1}^\infty d(d-1)^{k-1} \al_k^2 < \infty$.
\end{definition}
These linear factors are clearly $\Aut(T_d)$-invariant.
Furthermore, the random variable $X_v$ defined in \eqref{eq:lin} is always a centered Gaussian.
\begin{definition} \label{def:gaussian_process}
We call a collection of random variables $X_v$, $v \in V(T_d)$
a \emph{Gaussian process on $T_d$} if they are jointly Gaussian
and each $X_v$ is centered.
(Random variables are said to be jointly Gaussian
if any finite linear combination of them is Gaussian.)

Furthermore, we say that a Gaussian process $X_v$ is invariant if
it is $\Aut(T_d)$-invariant, that is, for arbitrary graph automorphism
$\Phi: V(T_d) \to V(T_d)$ of $T_d$ the joint distribution of
the Gaussian process $X_{\Phi(v)}$ is the same as that of the original process.
\end{definition}

The following invariant processes will be of special interest for us.
\begin{introthm} \label{thm:gaussian_ev}
For any real number $\la$ with $|\la| \leq d$
there exists a non-trivial invariant Gaussian process $X_v$ on $T_d$
that satisfies the eigenvector equation with eigenvalue $\la$,
i.e., (with probability $1$) for every vertex $v$ it holds that
$$ \sum_{u \in N(v)} X_u = \la X_v ,$$
where $N(u)$ denotes the set of neighbors of $v$.

The joint distribution of such a process is unique
under the additional condition that the variance of $X_v$ is $1$.
We will refer to this (essentially unique) process as
the \emph{Gaussian wave function} with eigenvalue $\la$.
\end{introthm}
%

These Gaussian wave functions can be approximated
by linear factor of i.i.d.\ processes
provided that $|\la| \leq 2\sqrt{d-1}$.
\begin{introthm} \label{thm:appr}
For any real number $\la$ with $|\la| \leq 2\sqrt{d-1}$
there exist linear factor of i.i.d.\ processes that converge in distribution
to the Gaussian wave function corresponding to $\la$.
\end{introthm}

The Gaussian wave function for negative $\lambda$ 
has negative correlations for neighbors. 
The set where the process takes values below a threshold $\tau$ 
is a percolation process, which -- with the right choice of parameters -- 
has high density but no infinite clusters. 
We will use this percolation to construct independent sets. 
Our first step, of independent interest, 
is to bound the critical threshold for this percolation. 
\begin{introthm} \label{thm:fin_comp}
Let $X_v$, $v \in T_3$ be the Gaussian wave function with eigenvalue $\la=-2\sqrt{2}$, 
and consider the percolation $S_\tau = \left\{ v \in V(T_d) \ : \ X_v \leq \tau \right\}$. 
If $\tau \leq 0.086$, then each cluster of $S_\tau$ is finite almost surely. 
(Note that for $\tau = 0.086$ the density of the percolation is above $0.534$, 
yet the clusters are finite almost surely.)
\end{introthm}
Asymptotically, for large values of $d$, 
Gamarnik and Sudan \cite{gamarnik} have recently showed that 
factor of i.i.d.\ processes can only produce independent sets 
with size at most $1/2+1/\sqrt{8}$ times the largest in random regular graphs. 
This means that upper bounds coming from random regular graphs 
(such as the Bollob\'as and McKay bounds) 
cannot be matched by factor of i.i.d.\ algorithms. 

For $d=3$, it is an open problem 
whether the best asymptotic independence ratio 
can be achieved with factor-of-i.i.d.\ algorithms such as ours. 

The rest of the paper is organized as follows:
in Section \ref{sec:2} we prove Theorems \ref{thm:gaussian_ev} and \ref{thm:fin_comp}, 
and derive other useful properties of Gaussian wave functions, 
in Section \ref{sec:3} we give a proof for Theorem \ref{thm:appr}, 
and in Section \ref{sec:4} we show how one can use these 
random processes to find large independent sets.



\section{Gaussian wave functions} \label{sec:2}

We call the random variables $X_v$, $v \in V(T_d)$ a Gaussian process
if they are jointly Gaussian and each $X_v$ is centered
(see Definition \ref{def:gaussian_process}).
The joint distribution is completely determined
by the covariances $\cov(X_u,X_v)$, $u,v \in V(T_d)$.
A Gaussian process with prescribed covariances exists if and only if
the corresponding infinite ``covariance matrix'' is positive semidefinite.

From this point on, all the Gaussian processes considered will be $\Aut(T_d)$-invariant.
For an invariant Gaussian process $X_v$
the covariance $\cov(X_u,X_v)$ clearly depends
only on the distance $d(u,v)$ of $u$ and $v$.
(The distance between the vertices $u,v$ is
the length of the shortest path connecting $u$ and $v$ in $T_d$.)
Let us denote the covariance corresponding to distance $k$ by $\si_k$.
So an invariant Gaussian process is determined (in distribution)
by the the sequence $\si_0, \si_1, \ldots$ of covariances.

Theorem \ref{thm:gaussian_ev} claims that for any $|\la| \leq d$
there exists an invariant Gaussian process that satisfies the eigenvector equation
$ \sum_{u \in N(v)} X_u = \la X_v $ for each vertex $v$.
What would be the covariance sequence of such a \emph{Gaussian wave function}?
Let $v_1, \ldots, v_d$ denote the neighbors of an arbitrary vertex $v_0$. Then
$$ 0 = \cov\left( X_{v_0}, 0 \right) =
\cov\left( X_{v_0}, X_{v_1} + \cdots + X_{v_d} - \la X_{v_0} \right) =
d \si_1 - \la \si_0 .$$
Also, if $u$ is at distance $k$ from $v_0$,
then it has distance $k-1$ from one of the neighbors $v_1,\ldots,v_d$,
and has distance $k+1$ from the remaining $d-1$ neighbors of $v_0$. Therefore
$$ 0 = \cov\left( X_{u}, 0 \right) =
\cov\left( X_{u}, X_{v_1} + \cdots + X_{v_d} - \la X_{v_0} \right) =
(d-1) \si_{k+1} + \si_{k-1} - \la \si_k .$$
After multiplying our process with a constant
we may assume that the variance of $X_v$ is $1$, that is, $\si_0 = 1$.
So the covariances satisfy the following linear recurrence relation:
\begin{equation} \label{eq:cov_rec}
\si_0 = 1; \ d \si_1 - \la \si_0 = 0; \
(d-1) \si_{k+1} - \la \si_k + \si_{k-1} = 0 ,\ k \geq 1.
\end{equation}
There is a unique sequence $\si_k$ satisfying the above recurrence.
Therefore to prove the existence of the Gaussian wave function
we only need to check that the corresponding infinite matrix is positive semidefinite.
This does not seem to be a straightforward task, though,
so we take another approach instead,
where we recursively consruct the Gaussian wave function.
(This approach will also yield some interesting and useful properties
of Gaussian wave functions, see Remark \ref{rm:ind} and \ref{rm:3reg_gen}.)
\begin{remark}
The case $|\la| \leq 2\sqrt{d-1}$ also follows
from the results presented in the next section,
where we construct factor of i.i.d.\ processes,
the covariance matrices of which converge to
the ``covariance matrix'' of the (supposed) Gaussian wave function.
As the limit of positive semidefinite matrices,
this ``covariance matrix'' is positive semidefinite, too,
and thus the Gaussian wave function indeed exists.
\end{remark}
\begin{proof}[Proof of Theorem \ref{thm:gaussian_ev}]
Let $\si_k$ be the solution of the recurrence relation \eqref{eq:cov_rec},
in particular,
$$ \si_0 = 1; \ \si_1 = \frac{\la}{d}; \ \si_2 = \frac{\la^2-d}{d(d-1)} .$$
We need to find a Gaussian process $X_v$, $v \in V(T_d)$ such that
\begin{equation} \label{eq:prescr}
\cov( X_u, X_v ) = \si_{d(u,v)}
\end{equation}
holds for all $u,v \in V(T_d)$.

We will define the random variables $X_v$ recursively
on larger and larger connected subgraphs of $T_d$.
Suppose that the random variables $X_v$ are already defined for $v \in S$
such that \eqref{eq:prescr} is satisfied for any $u,v \in S$,
where $S$ is a (finite) set of vertices for which the induced subgraph $T_d[S]$ is connected.
Let $v_0$ be a leaf (i.e., a vertex with degree $1$) in $T_d[S]$,
$v_d$ denotes the unique neighbor of $v_0$ in $T_d[S]$,
and $v_1, \ldots, v_{d-1}$ are the remaining neighbors in $T_d$.
We now define the random variables $X_{v_1}, \ldots, X_{v_{d-1}}$.
Let $(Y_1, \ldots, Y_{d-1})$ be a multivariate Gaussian
that is independent from $X_v$, $v \in S$ and
that has a prescribed covariance matrix that we will specify later. Set
$$ X_{v_i} \mathdef \frac{\la}{d-1} X_{v_0} - \frac{1}{d-1} X_{v_d} + Y_i
\ ,\ i=1,\ldots,d-1 .$$
For $1 \leq i \leq d-1$ we have
$$ \cov\left( X_{v_i}, X_{v_0}\right) = \frac{\la}{d-1} - \frac{1}{d-1} \si_1 = \si_1 ,$$
and if $u \in S \sm \{ v_0 \}$ is at distance $k \geq 1$ from $x_0$, then
$$ \cov\left( X_{v_i}, X_u \right) =
\frac{\la}{d-1} \si_k - \frac{1}{d-1} \si_{k-1} = \si_{k+1} .$$
We also need that
\begin{equation} \label{eq:needed}
\var\left( X_{v_i} \right) = \si_0 \ \mbox{ and }
\cov\left( X_{v_i}, X_{v_j}\right) = \si_2
\mbox{, whenever } 1 \leq i,j \leq d-1,\ i \neq j .
\end{equation}
Since
\begin{align*}
\var\left( X_{v_i} \right) &= \left( \frac{\la}{d-1} \right)^2 +
\left( \frac{1}{d-1} \right)^2 - \frac{2\la}{(d-1)^2} \si_1 + \var(Y_i) \mbox{ and }\\
\cov\left( X_{v_i}, X_{v_j}\right) &= \left( \frac{\la}{d-1} \right)^2 +
\left( \frac{1}{d-1} \right)^2 - \frac{2\la}{(d-1)^2} \si_1 + \cov(Y_i,Y_j) ,
\end{align*}
we can set $\var(Y_i)$ and $\cov(Y_i,Y_j)$
such that \eqref{eq:needed} is satisfied, namely let
$$ \var(Y_i) = a \mathdef \frac{(d-2)(d^2-\la^2)}{d(d-1)^2}
\mbox{ and } \cov(Y_i,Y_j) = b \mathdef \frac{-(d^2-\la^2)}{d(d-1)^2} .$$
We still have to show that there exist Gaussians $Y_1, \ldots, Y_{d-1}$
with the above covariances. The corresponding $(d-1) \times (d-1)$ covariance matrix
would have $a$'s in the main diagonal and $b$'s everywhere else.
The eigenvalues of this matrix are $a+(d-2)b$ and $a-b$
(with $a-b$ having multiplicity $d-2$).
Therefore the matrix is positive semidefinite
if $a \geq b$ and $a \geq -(d-2)b$.
It is easy to check that these inequalities hold when $|\la| \leq d$.
(In fact, $a = -(d-2)b$, so the covariance matrix is singular,
which means that there is some linear dependence between $Y_1, \ldots, Y_{d-1}$.
Actually, this linear dependence is $Y_1 + \cdots + Y_{d-1} = 0$,
and that is why the eigenvector equation $X_{v_1} + \cdots + X_{v_d} = \la X_{v_0}$ holds.)

So the random variables $X_v$ are now defined on the larger set
$S' = S \cup \{v_1, \ldots, v_{d-1} \}$ such that
\eqref{eq:prescr} is satisfied for any $u,v \in S'$. Since
$$
\begin{pmatrix}
  1 & \si_1  \\
  \si_1 & 1
\end{pmatrix}
$$
is positive semidefinite for $|\la| \leq d$, we can start with
a set $S$ containing two adjacent vertices, and then in each step
we can add the remaining $d-1$ vertices of a leaf to $S$.
The statement then follows from the Kolmogorov extension theorem.
\end{proof}
\begin{remark}[\emph{Markov field property}] \label{rm:ind}
There is an important consequence of the proof above,
which we will make use of when we will be computing
the probability of certain configurations for
a particular Gaussian wave function in Section \ref{sec:4}.
Let $u$ and $v$ be adjacent vertices in $T_d$.
They cut $T_d$ (and thus the Gaussian wave function on it) into two parts.
Our proof yields that the two parts of the process are independent under
the condition $X_u = x_u; X_v = x_v$ for any real numbers $x_u, x_v$.
\end{remark}
\begin{remark} \label{rm:3reg_gen}
If $d=3$ and $\la=-2\sqrt{d-1}=-2\sqrt{2}$, then we have $Y_2 = -Y_1$ in the above proof
with $\var( Y_1 ) = a = 1/12$. So we can express $X_{v_1}$ and $X_{v_2}$
with the standard Gaussian $Z = 2\sqrt{3} Y_1$ as follows:
\begin{align*}
X_{v_1} &= -\sqrt{2} X_{v_0} - \frac{1}{2} X_{v_3} + \frac{1}{ 2\sqrt{3} } Z \ \mbox{ and}\\
X_{v_2} &= -\sqrt{2} X_{v_0} - \frac{1}{2} X_{v_3} - \frac{1}{ 2\sqrt{3} } Z .
\end{align*}
Note that $Z$ is independent from the random variables $X_v, v \in S$,
in particular, it is independent from $X_{v_0}$, $X_{v_3}$.
\end{remark}

\subsection{Percolation corresponding to Gaussian wave functions}

Let $X_v$, $v \in V(T_d)$ be some fixed invariant process on $T_d$.
For any $\tau \in \IR$ we define
$$ S_\tau \mathdef \left\{ v \in V(T_d) \ : \ X_v \leq \tau \right\} ,$$
that is, we throw away the vertices above some threshold $\tau$.
(If the random variables $X_v$ are independent,
then we get the Bernoulli site percolation.
Otherwise $S_\tau$ is a \emph{dependent percolation}.)
One very natural question about this random set $S_\tau$ is
whether its components are finite almost surely or not.
Clearly, there exists a \emph{critical threshold} $\tau_c \in [-\infty, \infty]$ such that
for $\tau < \tau_c$ the component of any given vertex is finite almost surely,
while if $\tau > \tau_c$, then any given vertex is in 
an infinite component with some positive probability.

First we explain why it would be extremely useful for us
to determine this critical threshold (or bound it from below).
Let $\tau$ be below the critical threshold $\tau_c$ and
let $I_\tau$ be the ``largest'' independent set contained by $S_\tau$.
More precisely, we choose the largest independent set in each of
the (finite) components of $S_\tau$ and consider their union.
If the largest independent set is not unique, then we choose one in some invariant way.
This way we get an invariant independent set $I_d$.
(Moreover, if $X_v$ can be approximated by i.i.d.\ factors, then so is $I_d$.)
Clearly, 
the larger $\tau$ is, the larger the independent set we get.
So we want to pick $\tau$ close to the critical threshold.

The next lemma provides a sufficient condition for the components to be finite
in the case when our process $X_v$ is a Gaussian wave function.
Let us fix a path in $T_d$ containing $m+2$ vertices for some positive integer $m$ 
and fix the values assigned to the first and second vertex: $x$ and $y$, respectively. 
The sufficient condition 
is roughly the following: for any $x,y \leq \tau$,
the conditional probability of the event 
that the random values assigned to the remaining $m$ nodes 
are also below $\tau$ is less than $1/(d-1)^m$. 
In fact, the only thing that we will use about Gaussian wave functions is
the Markov field property pointed out in Remark \ref{rm:ind}.
\begin{lemma} \label{lem:suff}
Let $X_v, v \in T_d$ be a Gaussian wave function on $T_d$
and let $v_{-1}, v_0, v_1, \ldots, v_m$ be any fixed path in $T_d$
containing $m+2$ vertices for some positive integer $m$.
Suppose that there exists a real number $c < 1/(d-1)^m$ such that
$$ P\left( X_{v_i} \leq \tau \ ,\ 1 \leq i \leq m |
X_{v_{-1}} = x_{v_{-1}}; X_{v_{0}} = x_{v_{0}} \right) < c $$
holds for any real numbers $x_{v_{-1}}, x_{v_0} \leq \tau$.
Then each component of
$$ S_\tau = \left\{ v \in V(T_d) \ : \ X_v \leq \tau \right\} \subset V(T_d) $$
is finite almost surely.
\end{lemma}
\begin{proof}
Let $u$ be an arbitrary vertex and let us consider the component of $u$ in $S_\tau$.
Let $s$ be any positive integer.
We want to count the number of vertices in the component
at distance $sm+1$ from $u$.
The number of such vertices in $T_d$ is $d (d-1)^{sm}$.
For any such vertex $w$ the path from $u$ to $w$ can be split into $s$ paths, 
each having $m+2$ vertices and each overlapping with the previous and the next one
on two vertices. The Gaussian wave function on such a path depends
on the previous pathes only through the first two (overlapping)
vertices of the path (see Remark \ref{rm:ind}).
Using this fact and the assumption of the lemma,
one can conclude that the probability that $w$ is in the component
is less than $c^s$. Consequently, the expected number of vertices
in the component at distance $sm+1$ from $u$ is
at most $d (d-1)^{sm} c^s$, which is exponentially small in $s$.
Thus, by Markov's inequality, the probability that the
component has at least one vertex at distance $sm+1$ is exponentially small, too.
It follows that each component must be finite with probability $1$.
\end{proof}

Now we will use the above lemma to give a lower bound
for the critical threshold in the case $d=3$, $\la=-2\sqrt{2}$.
\begin{proof}[Proof of Theorem \ref{thm:fin_comp}]
Let $X_v$, $v \in T_3$ be the Gaussian wave function with eigenvalue $\la=-2\sqrt{2}$. 
We need to prove that $S_\tau$ has finite components almost surely for $\tau = 0.086$. 
We will use Lemma \ref{lem:suff} with $m=2$.
Let us fix a path containing four vertices of $T_3$, 
we denote the random variables assigned to the
first, second, third, and fourth vertex of the path by $X$, $Y$, $U$, and $V$, respectively.
Let $x,y$ be arbitrary real numbers not more than $\tau$.
From now on, every event and probabilitiy
will be meant under the condition $X=x; Y=y$.
According to Remark \ref{rm:3reg_gen} there exist independent
standard normal random variables $Z_1, Z_2$ such that
\begin{align*}
U &= -\sqrt{2} y - \frac{1}{2} x + \frac{1}{2\sqrt{3}} Z_1 ;\\
V &= -\sqrt{2} U - \frac{1}{2} y + \frac{1}{2\sqrt{3}} Z_2 =
\frac{3}{2} y + \frac{1}{\sqrt{2}} x -
\frac{1}{\sqrt{6}} Z_1 + \frac{1}{2\sqrt{3}} Z_2 .
\end{align*}
Our goal is to prove that the probability of $U \leq \tau; V \leq \tau$
is less than $1/4$ for any fixed $x,y \leq \tau$.
If we increase $y$ by some positive $\Delta$,
and decrease $x$ by $2\sqrt{2} \Delta$ at the same time,
then $U$ does not change, while $V$ gets smaller, 
and thus the probability in question increases. 
Thus setting $y$ equal to $\tau$ and changing $x$ accordingly
always yield a higher probability.
So from now on we will assume that $y=\tau$. Then
\begin{align*}
U \leq \tau &\Leftrightarrow Z_1 \leq \sqrt{3}x+2\sqrt{6}\tau+2\sqrt{3} \tau ;\\
V \leq \tau &\Leftrightarrow -Z_1 + \frac{1}{\sqrt{2}} Z_2 \leq
- \sqrt{3}x - \frac{\sqrt{3}}{\sqrt{2}}\tau .
\end{align*}
We notice that the sum of the right hand sides does not depend on $x$:
$$ a \mathdef \tau \left( 2\sqrt{6} + 2\sqrt{3} - \frac{\sqrt{3}}{\sqrt{2}} \right) .$$
Therefore we have to maximize the following probability in $d_1$:
$$ P\left( Z_1 \leq d_1 ; Z_2/q \leq Z_1 + a - d_1 \right) \mbox{, where } q = \sqrt{2} .$$
This can be expressed as a two-dimensional integral:
\begin{equation} \label{eq:prob_int}
f(d_1) \mathdef
\int_{-\infty}^{d_1}
\int_{-\infty}^{q (z_1 + a - d_1)}
\frac{1}{2\pi}\exp\left( - \frac{z_1^2 + z_2^2}{2} \right)
 \, \mathrm{d} z_2  \, \mathrm{d} z_1 .
\end{equation}
To find the maximum of the function $f(d_1)$, we take its derivative,
which can be expressed using the cumulative distribution function $\Phi$
of the standard normal distribution:
\begin{multline*}
f'(d_1) =
\frac{1}{\sqrt{2\pi}} \exp\left( \frac{-d_1^2}{2} \right) \Phi(qa)
- \frac{\sqrt{1+q^2}}{\sqrt{2\pi}q} \exp\left( \frac{-d_2^2}{2} \right)
\Phi\left( \sqrt{1+q^2}a - d_2/q \right) ,\\
\mbox{where } d_2 = \frac{q}{\sqrt{1+q^2}}(a-d_1) .
\end{multline*}
The derivative has a unique root, belonging to the maximum of $f$.
Solving $f'(d_1) = 0$ numerically ($d_1 \approx 0.555487$),
then computing the integral \eqref{eq:prob_int} ($\approx 0.249958$)
shows that $\max f < 1/4$ as claimed. 
(Both finding the root of the derivative and 
computing the integral numerically are easy to do, 
and $\max f < 1/4$ can be made rigorous using simple error bounds.)
\end{proof}


\section{Approximation with factor of i.i.d.\ processes} \label{sec:3}

Our goal in this section is to prove Theorem \ref{thm:appr}:
there exist linear factor of i.i.d.\ processes
approximating (in distribution) the Gaussian wave function with eigenvalue $\la$
provided that $|\la| \leq 2\sqrt{d-1}$. This will follow easily from the next lemma.
\begin{lemma} \label{lem:appr_ev}
Let $|\la| \leq 2\sqrt{d-1}$ be fixed.
For a sequence of real numbers $\al_0, \al_1, \ldots$
we define the sequence $\de_0, \de_1, \ldots$ as
\begin{equation} \label{eq:delta}
\de_0 \mathdef d \al_1 - \la \al_0; \
\de_k \mathdef (d-1)\al_{k+1} - \la \al_k + \al_{k-1},\ k \geq 1 .
\end{equation}
Then for any $\eps > 0$ there exists a sequence $\al_k$ such that
$$ \al_0^2 + \sum_{k\geq 1} d(d-1)^{k-1} \al_k^2 = 1 \mbox{ and }
\de_0^2 + \sum_{k\geq 1} d(d-1)^{k-1} \de_k^2 < \eps .$$
We can clearly assume that only finitely many $\al_k$ are nonzero.
\end{lemma}
\begin{remark}
We can think of such sequences $\al_k$ as
invariant approximate eigenvectors on $T_d$.
Let us fix a root of $T_d$ and
write $\al_k$ on vertices at distance $k$ from the root.
Then the vector $f \in \ell_2( V(T_d) )$ obtained
is spherically symmetric around the root
(i.e., $f$ is invariant under automorphisms fixing the root).
Furthermore, $ \| f \|^2 = \al_0^2 + \sum_{k\geq 1} d(d-1)^{k-1} \al_k^2 $.

As for the sequence $\de_k$, it corresponds to
the vector $A_{T_d} f - \la f \in \ell_2( V(T_d) )$,
where $A_{T_d}$ denotes the adjacency operator of $T_d$.
Therefore $ \| A_{T_d} f - \la f \|^2 = \de_0^2 + \sum_{k\geq 1} d(d-1)^{k-1} \de_k^2 $.
So the real content of the above lemma is that for any $\eps > 0$
there exists a spherically symmetric vector $f \in \ell_2( V(T_d) )$ such that
$ \| f \| = 1$ and $\| A_{T_d} f - \la f \| < \eps$.

In the best scenario $\de_k = 0$ would hold for each $k$,
that is, $\al_k$ would satisfy the following linear recurrence:
\begin{equation} \label{eq:rec}
d \al_1 - \la \al_0 = 0;  (d-1)\al_{k+1} - \la \al_k + \al_{k-1}=0
,\ k \geq 1 .
\end{equation}
However, for a non-trivial solution $\al_k$ of the above recurrence
we always have $\al_0^2 + \sum_{k\geq 1} d(d-1)^{k-1} \al_k^2 = \infty$.
This follows from the fact that the point spectrum of $A_{T_d}$ is empty.
\end{remark}

First we show how Theorem \ref{thm:appr} follows from the above lemma.
\begin{proof}[Proof of Theorem \ref{thm:appr}]
Let $Z_v$, $v \in V(T_d)$ be independent standard normal random variables.
Let $\eps > 0$ and let $\al_k$ as in Lemma \ref{lem:appr_ev}.
Let $X_v$ be the linear factor of $Z_v$ with coefficients $\al_k$ as in \eqref{eq:lin}. Then
$$ \var(X_v) = \al_0^2 + \sum_{k\geq 1} d(d-1)^{k-1} \al_k^2 = 1.$$
Let $v_0$ be an arbitrary vertex with neighbors $v_1, \ldots, v_d$. It is easy to see that
$$ X_{v_1} + \ldots + X_{v_d} - \la X_{v_0} =
(d \al_1 - \la \al_0) Z_{v_0} +
\sum_{k=1}^\infty \sum_{u: d(v_0,u) = k}
\left( (d-1)\al_{k+1} - \la \al_k + \al_{k-1} \right) Z_u .$$
So $X_{v_1} + \ldots + X_{v_d}-\la X_{v_0}$ is also a linear factor
with coefficients $\de_k$ as defined in \eqref{eq:delta}.
Therefore the variance of $X_{v_1} + \ldots + X_{v_d}-\la X_{v_0}$
is $\de_0^2 + \sum_{k\geq 1} d(d-1)^{k-1} \de_k^2 < \eps$.

What can we say about the covariance sequence $\si_k$ of the Gaussian process $X_v$?
We have $\si_0 = 1$ and
$$ \left| d \si_1 - \la \si_0 \right|,
\left| (d-1) \si_{k+1} - \la \si_k + \si_{k-1} \right| \leq
\sqrt{ \var( X_u ) \var\left( X_{v_1} + \cdots + X_{v_d} - \la X_{v_0} \right) }
< \sqrt{\eps} .$$
In other words, the equations in \eqref{eq:cov_rec} hold
with some small error $\sqrt{\eps}$.
If $K$ is a positive integer and $\de > 0$ is a real number,
then for sufficiently small $\eps$ we can conclude that
for $k \leq K$ the covariance $\si_k$ is closer than $\de$
to the actual solution of \eqref{eq:cov_rec}.
It follows that if $\eps$ tends to $0$,
then the covariance sequence of $X_v$ pointwise converges
to the unique solution of \eqref{eq:cov_rec}.
It follows that $X_v$ converges
to the Gaussian wave function in distribution as $\eps \to 0$.
\end{proof}
\begin{proof}[Proof of Lemma \ref{lem:appr_ev}]
It is enough to prove the statement for $|\la|<2\sqrt{d-1}$,
the case $\la = \pm 2\sqrt{d-1}$ then clearly follows.
Excluding $\pm 2\sqrt{d-1}$ will spare us some technical difficulties.

Let $\be_k$ be a solution of the following recurrence
\begin{equation} \label{eq:rec2}
d \be_1 - \la \sqrt{d-1} \be_0 = 0;
\be_{k+1} - \frac{\la}{\sqrt{d-1}} \be_k + \be_{k-1}=0
,\ k \geq 1 .
\end{equation}
(This is the recurrence that we would get from \eqref{eq:rec}
had we made the substitution $\be_k = (d-1)^{k/2} \al_k$.)
Since $|\la| < 2\sqrt{d-1}$, the quadratic equation
$x^2 - \frac{\la}{\sqrt{d-1}} x + 1 = 0$
has two complex roots, both of norm $1$, which implies that
\eqref{eq:rec2} has bounded solutions. Set
\begin{equation} \label{eq:al_def}
\al_k \mathdef \ro^k (d-1)^{-k/2} \be_k
\end{equation}
for some positive real number $1/2 \leq \ro < 1$.
Since $\be_k$ is bounded, $\al_0^2 + \sum_{k\geq 1} d(d-1)^{k-1} \al_k^2$ is finite
for any $\ro < 1$. It is also easy to see that 
$\al_0^2 + \sum_{k\geq 1} d(d-1)^{k-1} \al_k^2$ 
tends to infinity as $\ro \to 1-$. Furthermore,
\begin{multline*}
\de_k = (d-1)\al_{k+1} - \la \al_k + \al_{k-1} =
(d-1)^{-(k-1)/2}\ro^{k} \left( \ro \be_{k+1} -
\frac{\la}{\sqrt{d-1}} \be_k + \ro^{-1} \be_{k-1} \right) = \\
(d-1)^{-(k-1)/2}\ro^{k} (
\underbrace{ \be_{k+1} - \frac{\la}{\sqrt{d-1}} \be_k + \be_{k-1} }_{0}
+(\ro-1)\be_{k+1} + (\ro^{-1}-1)\be_{k-1} ).
\end{multline*}
Thus
$$ \sum_{k \geq 1} d(d-1)^{k-1} \de_k^2 \leq
d \sum_{k \geq 1} \ro^{2k} \left( (\ro-1)\be_{k+1} + (\ro^{-1}-1)\be_{k-1} \right)^2 .$$
Using that $\ro^{-1}-1 = (1-\ro)/\ro\leq 2(1-\ro)$ and
the fact that $\be_k$ is bounded we obtain that
$$ \sum_{k \geq 1} d(d-1)^{k-1} \de_k^2 \leq C(1-\ro)^2 \sum_{k \geq 1} \ro^{2k} =
C(1-\ro)^2 \frac{\ro^2}{1-\ro^2} = C \frac{\ro^2}{1+\ro} (1-\ro) \leq C (1-\ro) ,$$
where $C$ might depend on $d$ and $\la$, but not on $\ro$.
Therefore the above sum tends to $0$ as $\ro \to 1-$.
A similar calculation shows that $\de_0 \to 0$, too.
Therefore $\de_0^2 + \sum_{k \geq 1} d(d-1)^{k-1} \de_k^2 \to 0$.
Choosing $\ro$ sufficiently close to $1$ and rescaling $\al_k$ completes the proof.
\end{proof}
%


\section{Independent sets} \label{sec:4}

In this section we explain how one can find large independent sets
in $d$-regular, large-girth graphs using the Gaussian wave functions on $T_d$
and their linear factor of i.i.d.\ approximations.

Let $X_v$ be a linear factor of i.i.d.\ process
on the $d$-regular tree $T_d$
that has only finitely many nonzero coefficients:
\begin{equation} \label{eq:lin2}
X_v = \sum_{k=0}^N \sum_{u: d(v,u) = k} \al_k Z_u
\mbox{, where } \al_0, \al_1, \ldots, \al_N \in \IR .
\end{equation}
We will present different ways to produce independent sets on $T_d$
using the random variables $X_v$.
In each case the decision whether a
given vertex $v$ is chosen for the independent set will depend
(in a measurable and invariant way) only on the values of
the random variables $X_u$, $|d(v,u)|<N'$, where $N'$ is some fixed constant.
Therefore the obtained random independent set
will be a factor of the i.i.d.\ process $Z_v$.
Moreover, whether a given vertex $v$ is chosen will
depend only on the values in the $N+N'$-neighborhood of $v$.
It follows that the same random procedure can be
carried out on any $d$-regular finite graph
provided that its girth is sufficiently large,
and the probability that a given vertex is chosen will be the same.


So we can work on $T_d$ (instead of graphs with sufficiently large girth).
We will choose the coefficients $a_0, \ldots, a_N$ in such a way
that the process \eqref{eq:lin2} approximates the Gaussian wave function
with eigenvalue $\la = -2\sqrt{d-1}$ (see Theorem \ref{thm:appr}).
In the limit we can actually replace the underlying process $X_v$
with the Gaussian wave function.
So from this point on, let $X_v$, $v\in V(T_d)$ denote
the Gaussian wave function with eigenvalue $-2\sqrt{d-1}$.
We will define random independent sets on $T_d$
that are measurable and invariant functions of this process $X_v$.
Then the probability $p$ that $v$ is
in the independent set is the same for every vertex $v$.
We will call this probability $p$ the \emph{size} of
the random independent set.
(If we replace the underlying Gaussian wave function $X_v$
with an approximating process in the form \eqref{eq:lin2},
but otherwise use the same measurable and invariant way to
produce a random independent set from the underlying process,
then we get a factor of i.i.d.\ independent set with size arbitrarily close to $p$.
Once we work with processes like \eqref{eq:lin2},
we can carry out the procedure on finite regular graphs as well
provided that the girth is sufficiently large.
Thus for any $\eps > 0$ and for any $n$-vertex, $d$-regular graph $G$
with girth sufficiently large (depending on $\eps$)
we have a random independent set in $G$
with expected size at least $(p-\eps)n$.
It means that the $\liminf$ (as the girth goes to infinity)
of the independence ratio is at least $p$.)

Our method works best when the degree $d$ is equal to $3$.

\subsection{The $3$-regular case}

Let $d=3$, then $\la = -2\sqrt{d-1}=-2\sqrt{2}$ and the covariance sequence of $X_v$ is
$$ \si_0 = 1; \si_1 = \frac{-2\sqrt{2}}{3}; \si_2 = \frac{5}{6}; \ldots .$$

\bigskip

\noindent\textbf{First approach.} We choose those vertices $v$
for which $X_v > X_u$ for each neighbor $u \in N(v)$.

We need to compute the probability
$$ P\left( X_{v_0} > X_{v_1}; X_{v_0} > X_{v_2}; X_{v_0} > X_{v_3}\right) ,$$
where $v_0$ is an arbitrary vertex with neighbors $v_1, v_2, v_3$.
We will use the fact that if $(Y_1,Y_2,Y_3)$ is a non-degenerate multivariate Gaussian,
then the probability that each $Y_i$ is positive can be
expressed in terms of the pairwise correlations as follows:
\begin{equation} \label{eq:corr}
P\left( Y_1 > 0; Y_2 > 0 ; Y_3 >0 \right) =
\frac{1}{2} - \frac{1}{4\pi} \sum_{1\leq i < j \leq 3}
\arccos\left( \corr( Y_i, Y_j ) \right) .
\end{equation}
Indeed, the probability on the left can be expressed as
the standard Gaussian measure of the intersection of
three half-spaces through the origin.
This, in turn, equals the relative area of a spherical triangle
with angles $\pi - \arccos\left( \corr( Y_i, Y_j ) \right)$,
which is given by the standard formula \eqref{eq:corr}.

Let $Y_i = X_{v_0} - X_{v_i}$, $i=1,2,3$, then we have
$$ \corr(Y_1, Y_2) = \frac{ \cov(Y_1,Y_2) }{ \sqrt{\var(Y_1) \var(Y_2)} } =
\frac{ 1+\si_2-2\si_1 }{2-2 \si_1} = \frac{11+8\sqrt{2}}{12+8\sqrt{2}} =
\frac{1+2\sqrt{2}}{4} .$$
The two other correlations are the same, therefore
$$ P\left( v_0 \mbox{ is chosen} \right) =
\frac{1}{2} - \frac{3}{4\pi} \arccos \left( \frac{1+2\sqrt{2}}{4} \right) =
0.4298245... $$
So by simply choosing each vertex that is larger than its neighbors,
we get an independent set of size larger than $0.4298$.
Note that we could choose the vertices that are smaller than their neighbors
and would get an independent set of the same size.
Moreover, these two independent sets are clearly disjoint.


\bigskip

\noindent\textbf{Second approach.}
We fix some threshold $\tau \in \IR$ and we delete those vertices $v$
for which $X_v > \tau$, then we consider
the connected components of the remaining graph.
If a component is small (its size is at most some fixed $N'$),
then we choose an independent set of size at least half the size of the component.
We can do this in a measurable and invariant way.
For example, we partition the component into two independent sets
(this partition is unique, since each component is connected and bipartite),
if one is larger than the other, we choose the larger,
if they have equal size, we choose the one containing
the vertex with the largest value in the component.
If a component is large, then we simply do not choose any vertex from that component.
(The idea is to set the parameter $\tau$ in such a way that
the probability of large components is very small.)

We used a computer to simulate the procedure described above.
Setting $\tau = 0.12$ and $N'=200$ the simulation showed that
the probability that a given vertex is chosen is above $0.438$.
In what follows we will provide rigorous
(but -- in the case of the best result -- computer-assisted)
estimates of this probability.

From this point on, we will assume that $\tau$ is below the critical threshold,
that is, each component is finite almost surely.
It follows that with probability arbitrarily close to $1$
the component of any given vertex has size at most $N'$
provided that $N'$ is sufficiently large.
Let $p_s$ denote the probability that the component of a given vertex has size $s$.
(If a vertex is deleted, then we say that its component has size $0$.
Thus $p_0$ is simply the probability that $X_v > \tau$.)
If a component has size $2k-1$ for some $k \geq 1$,
then we choose at least $k$ vertices from the component.
If a component contains an even number of vertices,
then we choose at least half of the vertices.
Thus the probability that a vertex is chosen
(in the limit as $N' \to \infty$) is at least
\begin{equation} \label{eq:prob_sum}
\sum_{k=1}^\infty \frac{k}{2k-1} p_{2k-1} +
\frac{1}{2}\left( 1 - p_0 - \sum_{k=1}^\infty p_{2k-1} \right) =
\frac{1}{2}(1-p_0) + \sum_{k=1}^\infty \frac{1}{2(2k-1)} p_{2k-1} .
\end{equation}
The main difficulty in this approach is to determine (or estimate)
the probabilities $p_{2k-1}$
(each can be expressed as an integral of a multivariate Gaussian,
where the domain of the integration is an unbounded polyhedron).
These integrals can be computed with high precision using a computer (up to $p_5$).

\smallskip

\noindent$\boxed{\tau = 0}$
First we discuss what bound can be obtained
with no computer assistance whatsoever.
If we set $\tau=0$, then clearly $p_0 = 1/2$.
We can even compute the exact value of $p_1$. We notice that
$X_{v_1}>0$, $X_{v_2}>0$ and $X_{v_3}>0$ imply that $X_{v_0} < 0$,
because we have a Gaussian wave function with negative eigenvalue.
Thus using \eqref{eq:corr} we obtain
\begin{multline*}
p_1 = P \left( X_{v_0} \leq 0; X_{v_1}>0; X_{v_2}>0; X_{v_3}>0 \right) =
P \left( X_{v_1}>0; X_{v_2}>0; X_{v_3}>0 \right) = \\
\frac{1}{2} - \frac{3}{4\pi} \arccos \left( \corr( X_{v_1}, X_{v_2} ) \right) =
\frac{1}{2} - \frac{3}{4\pi} \arccos \left( \frac{5}{6} \right) .
\end{multline*}
Using this and the trivial estimates $p_{2k-1} > 0$ for $k \geq 2$,
\eqref{eq:prob_sum} yields the following lower bound:
$$ \frac{1}{2} - \frac{3}{8\pi} \arccos \left( \frac{5}{6} \right) = 0.4300889... $$
As far as the authors know, this is the best bound
that is not computer-aided.

Doing the same for vertices above the threshold (i.e., vertices with positive values)
clearly results in an another independent set
that has the same size and that is disjoint from the other independent set.
The induced subgraph on the union of these two disjoint independent sets
is a bipartite graph. This proves Theorem \ref{thm:ind_bipartite}.

\smallskip

\noindent $\boxed{\tau = 0.086}$
Here we discuss how we obtained the bound
$0.4361$ stated in Theorem \ref{thm:3reg}.
We set $\tau = 0.086$
(the largest $\tau$ for which we know
the components to be finite almost surely, see Theorem \ref{thm:fin_comp}).
Then $p_0 = 1-\Phi(0.086) = 0.46573321...$, where $\Phi$ is
the cumulative distribution function of the standard Gaussian.
Given a fixed path containing $s$ vertices of $T_3$,
$p'_s$ denotes the probability that the path is a component.

For any $k \geq 2$ the number of paths with $2k-1$ vertices
through any given vertex is $(2k-1) \cdot 3 \cdot 4^{k-2}$.
Furthermore, a component with $3$ vertices must be a path,
therefore we have the following relations between $p_{2k-1}$ and $p'_{2k-1}$:
\begin{equation} \label{eq:p&p'}
p_1 = p'_1; \ p_3 = 9 p'_3; \ 
p_{2k-1} \geq (2k-1) \cdot 3 \cdot 4^{k-2} p'_{2k-1}, \ k\geq 2 .
\end{equation}
As explained in the appendix, 
the probabilities $p'_s$ can be expressed as integrals. 
Although the occurring integrals cannot be computed analytically, 
the approximate values of $p'_1$, $p'_3$, and $p'_5$ 
can be determined by numerical integration: 
$$ p'_1 \approx 0.3272861614; \ 
p'_3 \approx    0.0025551311; \ 
p'_5 \approx    0.0002640467 .$$
Therefore
$$ p_1 = p'_1 \approx 0.3272861614; \ 
p_3 = 9 p'_3 \approx  0.0229961799; \ 
p_5 \geq 60 p'_5 \approx 0.0158428 .$$
Then the resulting lower bound for \eqref{eq:prob_sum} 
\begin{equation} \label{eq:rlb}
\frac{1}{2}(1-p_0) + \frac{1}{2}p_1 + \frac{1}{6}p_3 + \frac{1}{10} p_5 \geq 
0.5 (1-p_0) + 0.5 p'_1 + 1.5 p'_3 + 6 p'_5 \approx 0.43619355 . 
\end{equation}
We proved that the overall error is less than $0.000082$,
therefore the obtained bound is certainly above $0.4361$.
(See the appendix for details on the numerical integration and the error bound.)
\begin{remark}
The same numerical integrations can be carried out when $\tau = 0$,  
and thus one can get non-trivial estimates for $p_3$ and $p_5$ in that case, too. 
This way the bound in Theorem \ref{thm:ind_bipartite} can actually be improved to $0.868$. 
\end{remark}

\subsection{The $d \geq 4$ case}

The methods presented above for finding independent sets in $T_3$
work for regular trees with higher degree, too.
However, computing the occurring integrals even numerically
(with the required precision) seems very hard.
According to our computer simulation the second approach
with $\tau = 0.04$ would yield a lower bound $0.3905$ for $d=4$,
but we cannot prove this bound rigorously.
Note that the current best bound is $0.3901$
\cite[Table 5.3.1]{hoppen_thesis}.
When the degree is higher than $4$, our approach is not
as efficient as previous approaches in the literature.


\section{Appendix}

Let us consider the Gaussian wave function with eigenvalue
$\la = -2\sqrt{2}$ on the $3$-regular tree $T_3$.
We delete the vertices with value more than $\tau$
for some fixed positive real number $\tau$
and consider the components of the remaining vertices.
For an integer $s$ let $p_s$ denote the probability
that the component of a given vertex has size $s$.
These probabilities were used in Section \ref{sec:4} to bound
the independence ratio of $3$-regular, large-girth graphs, see \eqref{eq:prob_sum}.
Therefore, to get actual bounds, we need to determine (or estimate) $p_s$.
In what follows we will explain how $p_s$ can be expressed as an integral
in a way that the integration can be performed numerically with high precision
(at least for small integers $s$).

\subsection{Expressing $p_s$ as integrals}

Let $k \geq 0$ be an integer and let us fix
a path in $T_3$ with $k+2$ vertices: $v_0,v_1, \ldots, v_{k+1}$.
For $1 \leq i \leq k$ the neighbor of $v_i$
different from $v_{i-1}$ and $v_{i+1}$ is denoted by $v'_i$,
while the two neighbors of $v_0$ different from $v_1$ are $v'_0$ and $v''_0$.
The random variables (in the Gaussian wave function)
assigned to $v_i$, $v'_i$ and $v''_i$ will be
denoted by $X_i$, $X'_i$ and $X''_i$, respectively.
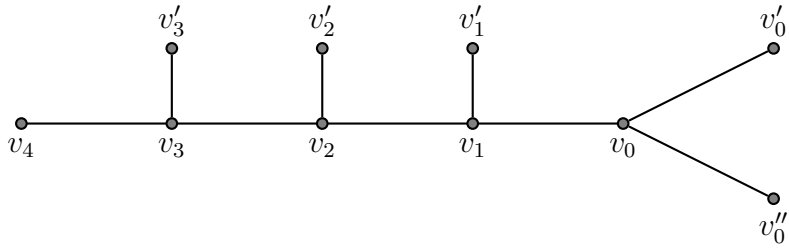
\begin{center}
\tikzstyle{every node}=[circle, draw, fill=black!50,inner sep=0pt, minimum width=4pt]
\begin{tikzpicture}[thick,scale=1.0]

\draw (0cm,0) node (v4) [label=below:$v_4$] {} 
  -- (2cm,0) node (v3) [label=below:$v_3$] {} 
  -- (4cm,0) node (v2) [label=below:$v_2$] {}
  -- (6cm,0) node (v1) [label=below:$v_1$] {} 
  -- (8cm,0) node (v0) [label=below:$v_0$] {} 
     (2cm,1cm) node (v3') [label=above:$v'_3$] {} 
     (4cm,1cm) node (v2') [label=above:$v'_2$] {} 
     (6cm,1cm) node (v1') [label=above:$v'_1$] {} 
     (10cm,1cm) node (v0') [label=above:$v'_0$] {} 
     (10cm,-1cm) node (v0'') [label=below:$v''_0$] {};

\draw
  (v0) -- (v0')
  (v0) -- (v0'')
  (v1) -- (v1')
  (v2) -- (v2')
  (v3) -- (v3');

\end{tikzpicture}

\end{center}
We define the function $f_k:\IR^2 \to [0,1]$ as the following conditional probability:
$$ f_k(x_{k+1}, x_k) =
P\left( X_i \leq \tau, 0 \leq i \leq k-1; X'_i > \tau, 0 \leq i \leq k ;
X''_0 > \tau | X_{k+1} = x_{k+1} ; X_k = x_k \right) .$$
The figure below shows the case $k=3$.
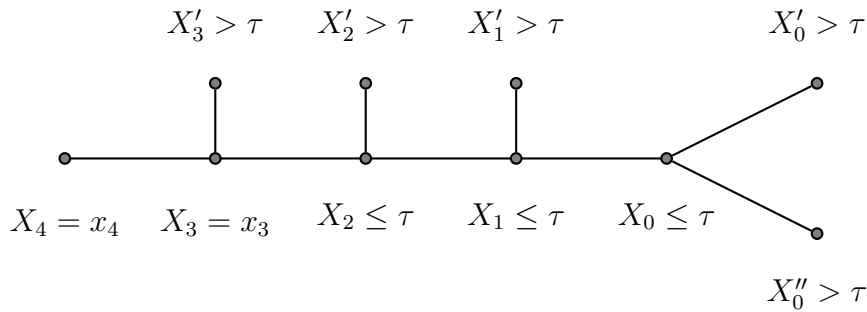
\begin{center}
\tikzstyle{every node}=[circle, draw, fill=black!50,inner sep=0pt, minimum width=4pt]
\begin{tikzpicture}[thick,scale=1.0]

\draw (0cm,0) node (v4) [label=below:$X_4 \eq x_4$] {} 
  -- (2cm,0) node (v3) [label=below:$X_3 \eq x_3$] {} 
  -- (4cm,0) node (v2) [label=below:$X_2 \leq \tau$] {}
  -- (6cm,0) node (v1) [label=below:$X_1 \leq \tau$] {} 
  -- (8cm,0) node (v0) [label=below:$X_0 \leq \tau$] {} 
     (2cm,1cm) node (v3') [label=above:$X'_3 > \tau$] {} 
     (4cm,1cm) node (v2') [label=above:$X'_2 > \tau$] {} 
     (6cm,1cm) node (v1') [label=above:$X'_1 > \tau$] {} 
     (10cm,1cm) node (v0') [label=above:$X'_0 > \tau$] {} 
     (10cm,-1cm) node (v0'') [label=below:$X''_0 > \tau$] {};

\draw
  (v0) -- (v0')
  (v0) -- (v0'')
  (v1) -- (v1')
  (v2) -- (v2')
  (v3) -- (v3');

\end{tikzpicture}

\end{center}
There is a recursive integral formula for these functions.
According to Remark \ref{rm:3reg_gen} there exists a standard Gaussian $Z_k$
independent from $X_{k+1}$, $X_k$ such that
\begin{align*}
X_{k-1} &= -\sqrt{2} X_k - \frac{1}{2} X_{k+1} - \frac{1}{ 2\sqrt{3} } Z_k \ \mbox{ and}\\
X'_{k} &= -\sqrt{2} X_k - \frac{1}{2} X_{k+1} + \frac{1}{ 2\sqrt{3} } Z_k .
\end{align*}
This yields the following formula for
the conditional probability $f_k(x_{k+1}, x_k)$ for $k\geq 1$:
\begin{equation} \label{eq:fk}
f_k(x_{k+1}, x_k) = \int_{| 2\sqrt{6} x_k + \sqrt{3} x_{k+1} + 2\sqrt{3}\tau |}^{\infty}
\phi(z_k) f_{k-1}\left( x_k,
-\sqrt{2} x_k - \frac{1}{2} x_{k+1} - \frac{1}{ 2\sqrt{3} } z_k \right) \, \mathrm{d} z_k ,
\end{equation}
where $\phi(t) = e^{-t^2/2} / \sqrt{2\pi}$ is
the density function of the standard normal distribution.
As for the case $k=0$ (see the figure below), we have
$$ f_0(x_1, x_0) = \int_{2\sqrt{6} x_0 + \sqrt{3} x_{1} + 2\sqrt{3}\tau}
^{-(2\sqrt{6} x_0 + \sqrt{3} x_{1} + 2\sqrt{3}\tau)}
\phi(z_0) \, \mathrm{d} z_0 .$$
(We use the convention that $\int_a^b $ is $0$ whenever $a>b$.)
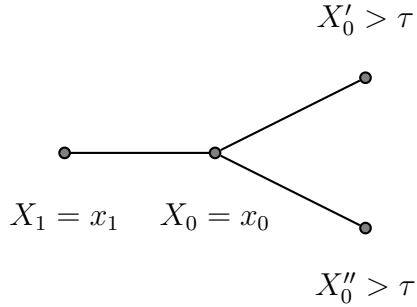
\begin{center}
\tikzstyle{every node}=[circle, draw, fill=black!50,inner sep=0pt, minimum width=4pt]
\begin{tikzpicture}[thick,scale=1.0]

\draw 
  (0cm,0) node (v1) [label=below:$X_1 \eq x_1$] {} 
  -- (2cm,0) node (v0) [label=below:$X_0 \eq x_0$] {} 
     (4cm,1cm) node (v0') [label=above:$X'_0 > \tau$] {} 
     (4cm,-1cm) node (v0'') [label=below:$X''_0 > \tau$] {};

\draw
  (v0) -- (v0')
  (v0) -- (v0'');

\end{tikzpicture}

\end{center}
So 
\begin{equation} \label{eq:f0}
f_0(x_1, x_0) = g_0(2\sqrt{6} x_0 + \sqrt{3} x_{1} + 2\sqrt{3}\tau) 
\mbox{, where }
g_0(t) = \left\{ 
\begin{array}{ll}
1-2\Phi(t) & \mbox{if } t<0\\
0 & \mbox{otherwise.}
\end{array}
\right.
\end{equation}
(Here $\Phi$ denotes the cumulative distribution function 
of the standard normal distribution.)

For a positive integer $s$ let us fix a path in $T_3$
containing $s$ vertices. Then $p'_s$ will denote
the probability that this path is a component, that is,
the values on the vertices of the path are all below $\tau$
and the values on all the adjacent vertices are above $\tau$.
(See \eqref{eq:p&p'} for the relation between $p_s$ and $p'_s$.)
In view of Remark \ref{rm:ind}, the probabilities $p'_s$
can be expressed with the functions $f_k$ as follows:
if $s \geq 2$, then for any integer $0 \leq m \leq s-2$
\begin{equation} \label{eq:p'}
p'_{s} =
\int_{-\infty}^\tau \int_{-\infty}^\tau \phi_2(u,v) f_m(u,v) f_{s-2-m}(v,u)
\, \mathrm{d} v \, \mathrm{d} u ,
\end{equation}
where $\phi_2$ is the density function of the $2$-dimensional
centered normal distribution with covariance matrix
$\begin{pmatrix}
  1 & \si_1  \\
  \si_1 & 1
\end{pmatrix}$, where $\si_1 = -2\sqrt{2}/3$.
As for $s=1$,
$$ p_1 = p'_1 = \int_{\tau}^\infty \int_{-\infty}^\tau \phi_2(u,v) f_0(u,v)
\, \mathrm{d} v \, \mathrm{d} u .$$
Our goal is to find the value of $p'_1$, $p'_3$ and $p'_5$.
Using \eqref{eq:p'} with $s=3, m=0$ and $s=5, m=1$:
\begin{align*}
p'_3 &= \int_{-\infty}^\tau \int_{-\infty}^\tau  \phi_2(u,v) f_0(u,v) f_1(v,u)
\, \mathrm{d} v \, \mathrm{d} u ,\\
p'_5 &= \int_{-\infty}^\tau \int_{-\infty}^\tau  \phi_2(u,v) f_1(u,v) f_2(v,u)
\, \mathrm{d} v \, \mathrm{d} u .
\end{align*}
%


\subsection{Numerical integration and bounding the error}

Next we explain how the above integrals 
(expressing $p'_1$, $p'_3$, and $p'_5$) 
can be computed numerically. 
We first have to compute the functions $f_0$, $f_1$, $f_2$. 
We will store their (approximate) values at the points of a fine grid, 
and we treat them as if they were $0$ outside some bounded region. 
Once we know $f_k$, the value of $f_{k+1}$ at each point can be obtained 
as a one-dimensional integral, see \eqref{eq:fk}. 
We divide the interval of integration into little pieces and 
on each piece $[x,x+\delta]$ we approximate the integral 
using the \emph{trapezoid rule}: 
$$\int_x^{x+\delta}f(t)\,\mathrm{d}t \approx 
\delta \frac{f(x)+f(x+\delta)}{2} .$$
When computing $f_{k+1}(\vec{x})$ at some point $\vec{x} \in \IR^2$ on our grid, 
we need the values of $f_{k}$ at points that are not on our grid. 
These values are interpolated from the values at the closest grid points 
in a bilinear way in the two coordinates. 
Once we have computed $f_0$, $f_1$, and $f_2$, 
the final (two-dimensional) integrals are calculated using 
the two-dimensional version of the trapezoid rule. 
The overall run-time is cubic in the resolution of the grid. 
We have to choose our grid carefully to get 
a reasonable run-time and reach the needed precision. 

Next, we explain how to estimate the numerical error, 
which comes from the following five sources:
\begin{itemize}
\item truncation of the region of integration,
\item error in the trapezoid rule,
\item using interpolated values of some functions,
\item floating point errors, and
\item errors carried over from previous integration. 
\end{itemize}

The function $f_0$ can be expressed 
in terms of the cumulative distribution function of 
the standard normal distribution $\Phi$, see \eqref{eq:f0}. 
According to \eqref{eq:fk} the value $f_{k+1}$ 
at some point $\vec{x} \in \IR^2$ is defined 
as a (one-dimensional) integral of the following form:
\begin{equation} \label{eq:fk_simple}
f_{k+1}(\vec{x}) = \int_{|\vec{c}^T\vec{x} + d|}^\infty
\phi(z)f_{k}(A\vec{x}+\vec{b}z) \,\mathrm{d} z,
\end{equation}
where $A$ is a $2 \times 2$ matrix, $\vec{b}$, $\vec{c}$ are two-dimensional vectors, 
and $d$ is a real number. 
It is clear from \eqref{eq:f0} that $ 0\le f_0(\vec{x})\le 1 $. 
It easily follows by induction that 
$$ 0\le f_k(\vec{x})\le 2^{-k} \mbox{ for all } 
\vec{x} \in \mathbb{R}^2 .$$ 
Thus when we change the interval of integration in \eqref{eq:fk_simple} 
to $\left[|\vec{c}^T\vec{x} + d|, R\right]$ for some $R>0$, 
we make an error less than the tail probability of 
a standard Gaussian, which can be bounded as follows: 
$$ \int_{R}^\infty \phi(z) \,\mathrm{d} z \le 
\frac{e^{-R^2/2}}{R\sqrt{2\pi}} .$$

Let us now turn to the error of the trapezoid rule. 
If $f$ is doubly differentiable on the interval $[x,x+\delta]$, then 
$$\int_x^{x+\delta}f(t)\,\mathrm{d}t = 
\delta \frac{f(x)+f(x+\delta)}{2} - \frac{\delta^3}{12}f''(\xi) $$
for some $\xi \in [x,x+\delta]$, see \cite[p.~216]{atkinson}. 
So whenever we have a good uniform bound for $|f''|$ on the interval $[x,x+\delta]$, 
the trapezoid rule gives a good approximation of the integral:
$$ \left| \int_x^{x+\delta}f(t)\,\mathrm{d}t - 
\delta \frac{f(x)+f(x+\delta)}{2} \right| \leq 
\frac{\delta^3}{12} \sup_{\xi \in [x,x+\delta]} \left|f''(\xi)\right| .$$
Unfortunately, in our case $f$ is not always twice differentiable:
the absolute value in the integration bound in \eqref{eq:fk_simple} 
causes the first derivative to jump. 
This, however, occurs ``rarely'' and for those intervals 
we may use the following weaker bound relying only on the first derivative: 
$$ \left| \int_x^{x+\delta}f(t)\,\mathrm{d}t - 
\delta \frac{f(x)+f(x+\delta)}{2} \right| \leq 
\frac{\delta^2}{3} \sup_{\xi \in [x,x+\delta]} \left|f'(\xi)\right| .$$

Before we can use these estimates, we need to bound 
the derivatives of $\phi(z)f_{k}(A\vec{x}+\vec{b}z)$ for any fixed $\vec{x}$. 
We start with the derivatives of $f_0$, $f_1$, and $f_2$. 
These are functions of two variables that are defined recursively by integrals. 
We have found the following uniform bounds 
for the $\ell_2$ norm of the gradient vector $\partial f_k$ 
and the $\ell_2 \to \ell_2$ operator norm of 
the Hessian matrix $Hf_k$ (whenever it exists): 
\begin{align*}
\sup_\vec{x}\|\partial f_0(\vec{x})\| &\le 4.2,&
\sup_\vec{x}\|\partial f_1(\vec{x})\| &\le 5.9,& \sup_\vec{x}\|\partial f_2(\vec{x})\| &\le 6.4,\\
\sup_\vec{x} \|Hf_0(\vec{x})\| &\le 13.1,& \sup_\vec{x}
\|Hf_1(\vec{x})\| &\le 96.1, & \sup_\vec{x} \|Hf_2(\vec{x})\| &\le 252.8.
\end{align*}
On the other hand, for any fixed $\vec{x}$ we have the following 
for the first and second derivative of $\phi(z)f_{k}(A\vec{x}+\vec{b}z)$ 
(with respect to $z$):   
\begin{align*}
\left| \left( \phi(z)f_{k}(A\vec{x}+\vec{b}z) \right)' \right| \leq &
\left| \phi(z) \right| \cdot \left\| (\partial f_k)(A\vec{x}+\vec{b}z) \right\| 
\cdot \left\| \vec{b} \right\| +  
\left| \phi'(z) \right| \cdot \left| f_k(A\vec{x}+\vec{b}z) \right| ,\\
\left| \left( \phi(z)f_{k}(A\vec{x}+\vec{b}z) \right)'' \right| \leq &
\left| \phi(z) \right| \cdot \left\| (H f_k)(A\vec{x}+\vec{b}z) \right\| 
\cdot \left\| \vec{b} \right\|^2 + \\
 & 2 \left| \phi'(z) \right| \cdot \left\| (\partial f_k)(A\vec{x}+\vec{b}z) \right\| 
\cdot \left\| \vec{b} \right\| 
+ \left| \phi''(z) \right| \cdot \left| f_k(A\vec{x}+\vec{b}z) \right| .
\end{align*}
Now we are in a position to estimate the integration error 
when calculating $f_{k+1}(\vec{x})$ in \eqref{eq:fk_simple}. 
We have to add up errors of the trapezoid rule on every small interval. 
After replacing all $|f_k|$, $\|\partial f_k\|$, and $\|H f_k\|$ 
with the uniform bounds we have found, 
it remains to add up $\left|\phi(\xi_i)\right|$ 
where $\xi_i$ is a point from the $i$-th interval 
(and similarly for $\phi'$ and $\phi''$). 
Even though the points $\xi_i$ are not explicitly given, 
we can estimate these sums using the following simple observation: 
for a fixed interval length $\delta$ and 
a continuous function $f$ with total variation $V(f) < \infty$ we have
$$ \delta \sum_i |f(\xi_i)| \le \int |f| + \delta V(f) .$$

The interpolation errors can be treated similarly.
When we calculate $f_{k+1}$ at a grid point $\vec{x}$, 
we need values of $f_k$ at points $A\vec{x}+\vec{b}z$
where $z$ runs through the points of a one-dimensional grid. 
Since the values of $f_k$ are given only at the points of our grid, 
we need to interpolate from the values at those points. 
Errors coming from such interpolations can be easily bounded 
using the first and second derivatives, and 
thus can be treated the same way as errors of the trapezoid rule. 

If we carry out the above calculations using a computer, 
we get an approximate value for $f_k$ at each grid point $\vec{x}$. 
Let us denote this calculated version of $f_k$ by $\hat{f}_k$. 
We will also use the notation $\tilde{f}_k$, which is the calculated version of $f_k$ 
assuming that the computer ``knows'' $f_{k-1}$ precisely at each grid point 
and that the computer can make precise calculations with real numbers. 
The difference between $\tilde{f}_k$ and $f_k$ comes from 
the integration error, the interpolation error and 
the error coming from the fact that we are integrating on a finite interval. 
As for the difference between $\tilde{f}_k(\vec{x})$ and $\hat{f}_k(\vec{x})$, 
we have to take into account that we only know $f_{k-1}$ with some error 
($\hat{f}_{k-1}$) and such errors will be carried over to $\hat{f}_k(\vec{x})$. 
Moreover, we will also have computing errors (i.e., floating point errors). 
These types of errors are fairly easy to handle. 
Adding all these up, the obtained error bounds will 
depend only on the number of grid points $N$ 
(after fixing $R=7$ and $\tau = 0.086$): 
\begin{align*}
\sup_\vec{x}\left|\hat{f}_0(\vec{x})-f_0(\vec{x})\right| &\le 
\frac{20}{N^2} + 1.4\cdot 10^{-12} ;\\ 
\sup_\vec{x}\left|\hat{f}_1(\vec{x})-f_1(\vec{x})\right| &\le 
\frac{361}{N^2} + 2.0\cdot 10^{-12} ;\\
\sup_\vec{x}\left|\hat{f}_2(\vec{x})-f_2(\vec{x})\right| &\le 
\frac{1606}{N^2} + 2.3\cdot 10^{-12} .
\end{align*}

Once we have calculated the approximate values of $f_0$, $f_1$, and $f_2$ 
at the points of our grid, we are ready to compute 
the final two-dimensional integrals defining $p'_1$, $p'_3$, $p'_5$. 
Note that if we use the same grid for the two-dimensional numerical quadrature  
as we used for storing the values of the functions $f_k$, 
then no interpolation errors will occur. 
Other error terms can be treated along the same lines as in the one-dimensional case. 

Our original goal was to find the value of 
$0.5 (1-p_0) + 0.5 p'_1 + 1.5 p'_3 + 6 p'_5$, recall \eqref{eq:rlb}. 
Setting $N = 20000$ the numerical computations outlined above 
give $0.43619355$ and combining all the errors occurring 
we get the following upper bound for the overall error: 
$$ \frac{17094}{N^2} + 3.9\cdot 10^{-5} + \frac{5\cdot 10^5}{N^4} < 8.2\cdot 10^{-5} .$$

\subsection*{Acknowledgment}
\begin{tiny}
E.Cs.: This research was realized in the frames of T\'AMOP 4.2.4.\ 
A/1-11-1-2012-0001 ``National Excellence Program -- Elaborating and
operating an inland student and researcher personal support system''
The project was subsidized by the European Union and co-financed by
the European Social Fund. 
Research is partially supported by
European Research Council (grant agreement no.~306493),
MTA Renyi ``Lend\"ulet'' Groups and Graphs Research Group,
ERC Advanced Research Grant No.~227701,
KTIA-OTKA grant No.~77780.\\
V.H.~was supported by MTA R\'enyi ``Lend\"ulet'' Groups and Graphs Research Group.\\
B.V.~was supported by the Canada Research Chair and NSERC DAS programs.
\end{tiny}


\bibliographystyle{plain}	
\bibliography{refs}

\begin{thebibliography}{10}

\bibitem{atkinson}
K.~E. Atkinson.
\newblock {\em An introduction to numerical analysis}.
\newblock Wiley, 2nd edition, 1989.

\bibitem{Bo_ind_set}
B.~Bollob{\'a}s.
\newblock The independence ratio of regular graphs.
\newblock {\em Proc. Amer. Math. Soc.}, 83(2):433--436, 1981.

\bibitem{frieze_random_regular}
A.~M. Frieze and T.~{\L}uczak.
\newblock On the independence and chromatic numbers of random regular graphs.
\newblock {\em J. Combin. Theory Ser. B}, 54(1):123--132, 1992.

\bibitem{gamarnik}
D.~Gamarnik and M.~Sudan.
\newblock Limits of local algorithms over sparse random graphs.
\newblock {\em Preprint}, 2013.
\newblock arXiv:1304.1831v1.

\bibitem{hladky_bachelor}
J.~Hladk\'y.
\newblock Structural properties of graphs -- probabilistic and deterministic
  point of view.
\newblock {\em Bachelor Thesis, Charles University, Prague}, 2006.

\bibitem{hladky}
J.~Hladk\'y.
\newblock Induced bipartite subgraphs in a random cubic graph.
\newblock 2007.
\newblock \url{http://kam.mff.cuni.cz/~hladky/SVOC1.pdf}.

\bibitem{hoppen_thesis}
C.~Hoppen.
\newblock Properties with graphs of large girth.
\newblock {\em PhD Thesis, University of Waterloo}, 2008.

\bibitem{cubic}
F.~Kardo{\v{s}}, D.~Kr\'al, and J.~Volec.
\newblock Fractional colorings of cubic graphs with large girth.
\newblock {\em SIAM J. Discrete Math.}, 25(3):1454--1476, 2011.

\bibitem{lauer_wormald}
J.~Lauer and N.~Wormald.
\newblock Large independent sets in regular graphs of large girth.
\newblock {\em J. Combin. Theory Ser. B}, 97(6):999--1009, 2007.

\bibitem{mckay}
B.~D. McKay.
\newblock Independent sets in regular graphs of high girth.
\newblock {\em Ars Combin.}, 23A:179--185, 1987.

\bibitem{shearer}
J.~B. Shearer.
\newblock A note on the independence number of triangle-free graphs.
\newblock {\em Discrete Math.}, 46(1):83--87, 1983.

\bibitem{shearer2}
J.~B. Shearer.
\newblock A note on the independence number of triangle-free graphs, {II}.
\newblock {\em J. Combin. Theory Ser. B}, 53(2):300--307, 1991.

\end{thebibliography}

\end{document}